\documentclass[reqno, a4paper, 12pt]{amsart}
\usepackage{amsmath, amssymb, amsthm, enumitem}
\usepackage{placeins}

\pdfoutput=1
\usepackage{color}

\def\comment#1{\marginpar{\raggedright\scriptsize{\textcolor{red}{#1}}}}
\usepackage{hyperref}
\usepackage{txfonts}

\usepackage{tikz}
\usetikzlibrary{external}
\usetikzlibrary{backgrounds,patterns,decorations.pathreplacing,decorations.pathmorphing}
\usepackage{caption}
\usepackage{accents}

\newtheorem{theorem}{Theorem}
\newtheorem{assumption}[theorem]{Assumption}
\newtheorem{corollary}[theorem]{Corollary}
\newtheorem{definition}[theorem]{Definition}
\newtheorem{lemma}[theorem]{Lemma}
\newtheorem{proposition}[theorem]{Proposition}

{Important Convention}

\theoremstyle{remark}
\newtheorem{remark}[theorem]{Remark}

  \newcommand{\F}{\mathcal{F}}
 
 \newcommand{\M}{\mathsf{M}}
 \newcommand{\PP}{\mathbb{P}}
 \newcommand{\W}{\mathcal{W}}

 \renewcommand{\phi}{\varphi}

\newcommand{\E}{\mathbb{E}}
\newcommand{\C}{\mathcal{C}}
\renewcommand{\P}{\mathbb{P}}
\newcommand{\N}{\mathbb{N}}

\newcommand{\R}{\mathbb{R}}

\newcommand{\weakoptimal}{optimal weak }

\DeclareMathOperator{\supp}{supp}
\newcommand{\bes}{\begin{subequations}}
\newcommand{\ees}{\end{subequations}}
\newcommand{\eea}{\end{eqnarray}}

\renewcommand{\W}{{\mathbb W}}

\newcommand{\FF}{{\mathcal F}}

\renewcommand{\L}{{\mathcal L}}
\renewcommand{\SS}{\mathcal S}

\usepackage{bbm}

\renewcommand{\epsilon}{\varepsilon}

\DeclareMathOperator{\proj}{proj}

\newcommand{\fourIdx}[5]{%
\setbox1=\hbox{\ensuremath{^{#1}}}%
 \setbox2=\hbox{\ensuremath{_{#2}}}%
 \setbox5=\hbox{\ensuremath{#5}}%
 \hspace{\ifnum\wd1>\wd2\wd1\else\wd2\fi}%
 \ensuremath{\copy5^{\hspace{-\wd1}\hspace{-\wd5}#1\hspace{\wd5}#3}%
 _{\hspace{-\wd2}\hspace{-\wd5}#2\hspace{\wd5}#4}%
 }}

\numberwithin{equation}{section}
\numberwithin{theorem}{section}

\renewcommand{\subset}{\subseteq}

\renewcommand{\S}{S}%{\mathsf{S}}

\usepackage{subcaption}
\usepackage{graphicx}
\renewcommand{\mathrm}{}

\newcommand{\probref}[1]{{\normalfont (\nameref{#1})}}

\makeatletter
\newcommand{\mylabel}[2]{#2\def\@currentlabel{#2}\label{#1}}
\makeatother
\begin{document}

\begin{abstract}
Motivated by applications to geometric inequalities, Gozlan, Roberto, Samson, and Tetali \cite{GoRoSaTe17} introduced a transport problem for `weak' cost functionals. Basic results of optimal transport theory can be extended to this setup in {remarkable} generality. 

In this article we collect several problems from different areas that can be recast in the framework of weak transport theory, namely: the Schr\"odinger problem, the Brenier--Strassen theorem, optimal mechanism design,  linear transfers, semimartingale transport. Our viewpoint yields a unified approach and often allows to strengthen the original results. 
 
\medskip

\noindent\emph{keywords:} Schr\"odinger problem, Brenier--Strassen theorem,  linear transfers, semimartingale transport, optimal mechanism design, weak transport problem, duality, cyclical monotonicity.
\end{abstract}

\title{Applications of weak transport theory}
\author{J. Backhoff-Veraguas \and G. Pammer}
\maketitle

%\section{Introduction}
\section{Overview}
The optimal transport problem for weak costs was first introduced by  Gozlan{\color{black}, Roberto, Samson and Tetali} \ \cite{GoRoSaTe17} and has immediately generated interest in  several groups of researchers, see \cite{Sh16, Sh18, FaSh18, GoJu18, AlCoJo17, AlBoCh18} among others. To present the basic problem we introduce some notation. 
 Throughout $X$ and $Y$ denote Polish spaces. Given probability measures $\mu\in\mathcal P(X)$ and $\nu\in\mathcal P(Y)$ we write $\Pi(\mu,\nu)$ for the set of all couplings on $X\times Y$ with marginals $\mu$ and $\nu$.  Given a coupling $\pi$ on $X\times Y$ we denote a regular disintegration with respect to the first marginal by $(\pi_x)_{x\in X}$. 

We consider  cost functionals of the form 
$$C:X\times\mathcal P(Y)\to \mathbb R\cup\{+\infty\},$$
where $C$ is lower bounded and lower semicontinuous, and $C(x, \cdot)$ is assumed  to be convex on $\mathcal P (Y)$ for every $x\in X$. 
The weak transport problem is then to determine
\begin{align}\label{OWT}\tag{OWT}
	V_C(\mu,\nu) = \inf_{\pi \in\Pi(\mu,\nu)} \int_X C(x,\pi_x)\mu(dx).
\end{align}
The classical transport problem is included via $C(x, p) = \int c(x, y)\, d p(y)$ for a given cost function $c:X\times Y \to {\R \cup \{+\infty\}}$. 

{Fundamental} results in classical optimal transport theory include \emph{existence}, \emph{duality}, and \emph{cyclical monotonicity} of optimizers. Through a series of contributions (see \cite{GoRoSaTe17, AlBoCh18, GoJu18, BaBePa18}) it has been understood that these results extend in full generality to the weak transport setup. We will recall these results in some detail in Section 2 below. 

The purpose of the present article is to advertise {the power and flexibility of} \weakoptimal transport theory through the investigation of a broad variety of applications thereof. We will consider the following problems:

\begin{itemize}
\item 
 The \emph{Schr\"odinger problem} has recently received particular attention since it provides a regularized version of the transport problem that is numerically much more tractable than the classical counterpart.  On a technical level, the difference is that quadratic costs are replaced by the entropy wrt  a reference measure. Based on {monotonicity for weak transport costs} we give a short proof of the fundamental characterization of optimizers in the Schr\"odinger problem. (Section \ref{se:Schroedinger}.)
 
 \item The \emph{Brenier--Strassen Theorem} of  Gozlan and Juillet \cite{GoJu18} yields that $1$-Lipschitz maps that are gradients to convex functions are the optimizers of transport problems with  barycentric costs. This result plays a role in the probabilistic proof of the Caffarelli contraction theorem \cite{FaGoPr19}. We provide a short new derivation which is based on {monotonicity for weak transport costs} and emphasizes the similarity with the classical Brenier Theorem. (Section \ref{se:Brenier--Strassen}.)

\item A significant problem in the economics literature is to  \emph{optimize the revenue for a multiple good monopolist}.  The influential article \cite{DaDeTz17} of Daskalakis, Deckelbaum, and Tzamos suggests a systematic \linebreak framework to study such problems and provides a dual characterization. We link the multiple good monopolist problem to \weakoptimal transport and use the weak transport duality theorem to recover and strengthen the results of \cite{DaDeTz17}. (Section \ref{se:Daskalakis}.)

\item Bowles and Ghoussoub \cite{BoGh19} have recently introduced the class of \emph{linear transfers} which includes many specific couplings between  probability measures. A main theorem of \cite{BoGh19} provides the representation of linear transfers through weak transport under the assumption of compactness of the underlying spaces. We provide a short new derivation based on weak transport duality that allows to drop the compactness condition. In particular, this 
implies that the representation result {of Bowles and Ghoussoub} is valid for the important case of Euclidean space. (Section \ref{se:transfers}.)

\item The \emph{semimartingale transportation problem} was introduced by  Tan and Touzi \cite{TaTo13} and extends classical optimal transport to the case where mass  is transported along the trajectories of a semimartingale. As such, it also contains both martingale optimal transport \cite{GaLaTo,PHB,BeJu16}, and the drift control framework of Mikami and Thieullen \cite{MiTh06}, as particular cases. We show that weak transport theory can be used to strengthen the main duality result of \cite{TaTo13}. (Section \ref{se:TaTo}.)
%\item A specific instance of \emph{Strassen's theorem} asserts that there exists a two step martingale with marginals $\mu, \nu$ iff the marginals are in convex order. In the case of compact spaces, an abstract version of Strassen's theorem is given by Meyer's work on \emph{dilations}. We use weak transport theory to weaken the compactness assumption in Meyer's result. (Section \ref{se:dilations}.)

\end{itemize}

The paper is organized as follows. In Section \ref{se:literature} we give a brief overview of previous works connected to \weakoptimal transport. 
In Section \ref{se:framework} we review the basic results of \weakoptimal transport theory together with the necessary notation.  Then  we describe the various applications announced above in Sections {\ref{se:Schroedinger} - \ref{se:TaTo}.}

\section{Literature connected to the weak transport problem}\label{se:literature}
The weak transport problem was introduced by Gozlan, Roberto, Samson and Tetali \cite{GoRoSaTe17}, and shortly afterwards by Aliberti, Bouchitte and Champion \cite{AlBoCh18}. The problem has also been designated ``general transport problem'' and ``non-linear transport problem'' respectively. 

The initial works of Gozlan et al.\ \cite{GoRoSaTe17, GoRoSaSh18} 
are mainly motivated by applications to geometric inequalities. Indeed, particular costs of the form \eqref{OWT} were already considered by Marton \cite{Ma96concentration, Ma96contracting} and Talagrand \cite{Ta95, Ta96}.
The theory for problem {\eqref{OWT}} has been further developed in \cite{AlCoJo17, GoRoSaTe17, GoRoSaSh18, Sh16, Sa17, Sh18, FaSh18, GoJu18,  BaBePa18, BaPa19, BaBePa19}:
Basic results of existence and duality are established in the articles %\comment{MB: War die erste Dualitaet in \cite{GoRoSaTe17} oder \cite{GoRoSaSh18}?}
%\comment{GP: Sie war zuerst in \cite{GoRoSaTe17}} 
\cite{GoRoSaTe17, AlBoCh18, BaBePa18}. The notion of \emph{$C$-monotonicity} was developed  in \cite{BaBeHuKa17, GoJu18, BaBePa18} as an analogue of classical {$c$-cyclical montonicity} in order to provide a characterization of optimizers to the weak transport problem. A weak transport analogue to the case of quadratic costs in classical optimal transport, is the case of barycentric costs. This case has received particular attention, and we refer to \cite{GoRoSaSh18, Sh16, Sa17, Sh18, FaSh18, GoJu18,  BaPa19, BaBePa19, AlCoJo17}.

The weak transport viewpoint is useful for a number of  problems loosely related to stochastic optimization:  it appears in the recursive formulation of the causal transport problem \cite{BaBeLiZa16}, in \cite{AlCoJo17, AlBoCh18, BeJu17, BaBeHuKa17,GuMeNu17,CoVi19} it is used to investigate  martingale optimal transport problems,  {\color{black}in \cite{BaBaBeEd19a} it is applied to prove stability of pricing and hedging in mathematical finance}  and, as mentioned above, it appears in the recent probabilistic proof to the Caffarelli contraction theorem \cite{FaGoPr19}.

\section{{Fundamental} results of weak transport theory}\label{se:framework}
%\subsection{Notation}
 
For  $t \geq 1$,  $\mathcal P_t(X)$ denotes the set of Borel probability measures with finite $t$-th moment for some fixed  metric $d_X$ (compatible with the topology on $X$), i.e., a Borel probability measure $\mu$ is in $\mathcal P_t(X)$ iff for some $x_0 \in X$ we have
$$\int_X d_X(x,x_0)^t \mu(dx) < \infty.$$
The set of continuous functions on $X$ which are dominated by a multiple of $1 + d_X(x,x_0)^t$, is denoted by $\Phi_t(X)$.
We equip the set of probability measures $\mathcal P_t(X)$ with the $t$-th Wasserstein topology. Specifically, a sequence $(\mu_k)_{k\in\N}$ converges to $\mu \in \mathcal P_t(X)$ if  $\mu_k(f):=\int f d\mu_k$ converges to $\mu(f)$ for all $f\in\Phi_t(X)$. The space $\mathcal P(X)$ itself is equipped with the usual weak topology. {The same conventions apply to $Y$ instead of $X$.}

%A coupling is a probability measure on $X\times Y$ and we write $\pi \in \Pi(\mu,\nu)$ if the coupling $\pi$ has first marginal $\mu \in \mathcal P(X)$ and second marginal $\nu \in \mathcal P(X)$. We denote its $\mu$-uniquely defined disintegration wrt \ the first marginal by $(\pi_x)_{x\in X}$.

%\subsection{The key ingredient: Optimal Weak Transport}
%suggested a new class of (non-linear) optimal transport problems with cost functions of the form $C\colon X\times\mathcal P(X) \to \R\cup\{+\infty\}$. 
%An important feature of this class of cost function is the possibility to explicitly pose constraints on the kernel like order constraints, e.g., being a Martingale kernel.

%Under mild assumptions to $C$, we have pivotal results and tools in the spirit of classical optimal transport, like existence of optimizers, Kantorovich-type duality, and $C$-monotonicity.

We have already mentioned the basic assumptions on the function $C$ in the introductory section above. We recall it and make it more precise in the following definition. 
\begin{definition}[\textsc{A}]\label{prop:A}
	We say $C\colon X \times \mathcal P_t(Y) \to \R \cup \{+\infty\}$ satisfies property \probref{prop:A} iff
	\begin{enumerate}[label=\roman*)]
		\item $C$ is lower-semi continuous wrt \ the product topology on $X \times \mathcal P_t(Y)$,
		\item $C$ is bounded from below,
		\item the map $p \mapsto C(x,p)$ is convex, i.e., for all $x\in X$ and $p,q\in\mathcal P_t(Y)$ we have
		$$C(x,\lambda p + (1-\lambda) q) \leq \lambda C(x,p) + (1-\lambda) C(x,q)\quad \lambda \in [0,1].$$
	\end{enumerate}
\end{definition}
%
%The \emph{optimal weak transport problem} for a cost function $C$ satisfying \probref{prop:A} is defined by
%\begin{align}\label{OWT}\tag{OWT}
%	V_C(\mu,\nu) = \inf_{\pi \in\Pi(\mu,\nu)} \int_X C(x,\pi_x)\mu(dx).
%\end{align}

From now on until the end of this section, the cost function $C$ is assumed to satisfy property \probref{prop:A}.

We will need the following existence and continuity result from \cite{BaBePa18}: 
\begin{theorem}[Existence and {semi}continuity]\label{thm:existence}
	The infimum in \eqref{OWT} is attained and the value $V_C(\mu,\nu)$ depends in a lower semicontinuous way on the marginals $(\mu,\nu) \in \mathcal P(X) \times \mathcal P_t(Y)$.
\end{theorem}

In optimal transport, the renowned Kantorovich duality states that for lower semi-continuous and lower bounded cost functions $c\colon X\times Y \to \R\cup\{+\infty\}$ we have
\begin{align}\label{eq:classical kantorovich}
	\inf_{\pi \in \Pi(\mu,\nu)} \int_{X\times Y}c(x,y)\,\pi(dx,dy) = \sup_{\substack{f \in L_1(\mu),~g\in L_1(\nu),\\ f + g \leq c}} \mu(f) + \nu(g).
\end{align}
Here duality takes the following form, which resembles (and generalizes) \eqref{eq:classical kantorovich}, cf.  \cite[Theorem 3.1]{BaBePa18}:
\begin{theorem}[Kantorovich duality for weak transport]\label{thm:duality}
	The weak transport problem \eqref{OWT} admits the dual representation
\begin{align}\label{eq:weak kantorovich}
	\inf_{\pi \in \Pi(\mu,\nu)} \int_{X} C(x,\pi_x)\, \mu(dx) = \sup \mu(f) + \nu(g),
\end{align}
where the supremum is taken over functions $f \in L_1(\mu),~g\in{\Phi_t{(Y)}}$, satisfying 
$$f(x) + p(g) \leq C(x, p)$$ 
for $x\in X$, $p \in \mathcal P_t({Y})$. 
\end{theorem}
We will often use Theorem \ref{thm:duality} in the following (equivalent) form: 
	\begin{align}\label{eq:duality}
		\inf_{\pi \in \Pi(\mu,\nu)} \int_X C(x,\pi_x)\mu(dx) ={ \sup_{ g \in \Phi_{t}(Y)} -\nu(g) + \int_X R_Cg(x) \mu(dx),}
	\end{align}
	where ${ R_C g(x) = \inf_{p \in \mathcal P_t(Y)}  p(g) + C(x,p)}$. Moreover, the right-hand supremum in \eqref{eq:duality} can be restricted to functions in ${g \in \Phi_t(Y)}$ which are bounded from below.
	
	We also recall a further consequence of the proof of Theorem \ref{thm:duality} in {\cite{BaBePa18}} that provides further insight into the dual problem. The  convex conjugate of $\nu \mapsto V_C(\mu,\nu)$ admits a rather concrete representation: For any ${ g\in \Phi_t(Y)}$ we have
	\begin{align}\label{eq:duality convex conjugate}
		{\sup_{\nu \in \mathcal P_t(Y)} \nu(g) - V_C(\mu,\nu) = - \int_X R_C g(x) \mu(dx).}
	\end{align}

\medskip

The notion of $c$-cyclical monotonicity constitutes a necessary (and often also sufficient) optimality criterion for transport plans in classical optimal transport, i.e., for any measurable cost $c \colon X \times Y \to \R \cup \{-\infty, +\infty\}$ if $\pi^* \in \Pi(\mu,\nu)$ is an optimizer of $$V_c(\mu,\nu) = \inf_{\pi\in\Pi(\mu,\nu)} \int_{X\times Y} c(x,y)\pi(dx,dy),$$
where $|V_c(\mu,\nu)| < \infty$, then there exists $\Gamma \subset X \times Y$ with $\pi^*(\Gamma) = 1$ such that for all $N\in\N$ and $(x_k,y_k)_{k=1}^N$ in $\Gamma$ we have for any permutation $\sigma$ of $\{1,\ldots,N\}$ that
$$\sum_{k=1}^N c(x_k,y_k) \leq \sum_{k=1}^N c(x_k,y_{\sigma(k)}).$$
The importance of $c$-cyclical monotonicity has been understood at least since the publication of the seminal article \cite{GaMc96}. See \cite{Vi09, BiCa10, Be15} for minimal conditions that guarantee equivalence of optimality and $c$-cyclical \linebreak monotonicity. More recently, variants of this `monotonicity priniciple' have been applied in transport problems for finitely or infinitely many marginals, the martingale version of the transport problem, the Skorokhod embedding problem, and the distribution constrained optimal stopping problem, see \cite{Pa12fm, CoDeDi15, Gr16a, BeGr14, BeJu16, NuSt16, Za14, BeCoHu14} among others.

In the context of \weakoptimal transport the corresponding concept is $C$-monotonicity. Early versions can be found in \cite{BaBeHuKa17, BaBePa18, GoJu18}, while the following definition as well as the subsequent result are taken from  {\cite[Section 2]{BaPa19}}. 
\begin{definition}\label{def:C-monotonicity}
	A set $\Gamma \subset X \times \mathcal P_t(Y)$ is called $C$-monotone iff for any finite subset of points $(x_k,p_k)_{k=1}^N$ of $\Gamma$ we have
	$$\sum_{k=1}^N C(x,p_k) \leq \sum_{k=1}^N C(x,q_k)\quad \sum_{k=1}^N q_k = \sum_{k=1}^N p_k.$$
	A coupling $\pi$ with first marginal $\mu$ is called $C$-monotone iff there is a $C$-monotone set $\Gamma$ such that $(id_X,\delta_{\pi_x})_\# \mu$ is concentrated on $\Gamma$.
\end{definition}
\begin{theorem}[$C$-monotonicity]\label{thm:C-monotonicity}
	If \eqref{OWT} is finitely valued, then any optimizer is $C$-monotone.
\end{theorem}
	The reverse implication holds also true if $C$ is sufficiently regular, see \cite[Theorem 2.2]{BaPa19}. 
	
\begin{theorem}\label{thm:stability}
Assume that $\mu \in \mathcal P_t(X)$, $\nu \in \mathcal P_t(Y)$.
	If $C$ is continuous and {$|C(x,p)|\leq R\,(d_X(x,x_0)^t + \int_Y d_Y(y,y_0)^t p(dy))$ for some $R$,} then any $C$-monotone coupling $\pi \in \Pi(\mu,\nu)$ is optimal for \eqref{OWT}.
\end{theorem}

	\section{Structure of optimizers in the Schr\"odinger problem}\label{se:Schroedinger}

We refer the reader to \cite{Le14} for a survey on the Schr\"odinger problem / entropic transport problem (cf.\ \eqref{eq:rel ent problem} below). Starting with the articles \cite{Cu13,BeCaCuPe15,KuEc19}, the Schr\"odinger problem has received significant attention as a regularized, numerically tractable version of the classical transport problem.

	The main goal of this section is to recover, in Corollary \ref{Schroedinger_Char} below, the characterization of entropic cost optimal transport plans through the product structure of their density. % wrt the reference measure. %In this section we recover a classical result from entropic transport, cf.  \cite[Corollary 3.2]{Cs75}, by means of $C$-monotonicity. 
%\comment{\color{orange}In der Def von $H$ habe ich $p,q$ mit $\mu,\nu$ ausgetauscht.}
	Given a Polish space $Z$, the relative entropy of $\mu \in \mathcal P(Z)$ wrt a `reference measure'  $\nu\in \mathcal P(Z)$ is defined as
	$$H(\mu|\nu) = \begin{cases} \int_X \log\left(\frac{d \mu}{d \nu}\right) \mu(dx) & \mu\ll \nu, \\ +\infty & \text{else.} \end{cases}$$	
	%To avoid further technicalities we assume that the reference probability measure $\gamma$ is equivalent to the product measure $\mu\otimes \nu$.
\begin{theorem}\label{thm:rel ent} 
	Let $\gamma$ be a probability measure equivalent to the product measure $\mu \otimes \nu$ for $\mu \in \mathcal P(X)$ and $\nu \in \mathcal P(Y)$. If the coupling $\pi^*\in\Pi(\mu,\nu)$ is optimal for the problem
	\begin{align}\label{eq:weak rel ent problem}
	\inf_{\pi \in \Pi(\mu,\nu)} \int_X H(\pi_x|\gamma_x) \mu(dx)< +\infty,
	\end{align}
	then there is a measurable set $\Gamma\subset X \times \mathcal P(Y)$ with $\mu(\{x\in X\colon (x,\pi^*_x) \in \Gamma\}) = 1$ such that for all $(x,\pi^*_x),(z,\pi^*_z) \in \Gamma$ there is a constant $\alpha > 0$ so that
	\begin{align}\label{eq_pix_piz}
		\frac{d \pi^*_x}{d \gamma_x} = \alpha \frac{d \pi^*_z}{d \gamma_z}\quad \nu\text{-a.e.}
	\end{align}
\end{theorem}

\begin{proof}
	The minimization problem \eqref{eq:weak rel ent problem} constitutes a weak transport problem. Therefore, there is by Theorem \ref{thm:C-monotonicity} a measurable set $\Gamma \subset X \times \mathcal P(Y)$ such that $\pi^*$ is $C$-monotone on $\Gamma$. Moreover, we may assume w.l.o.g.\ that $\gamma_x \sim \nu$ for all $(x,\pi^*_x) \in \Gamma$ and $H(\pi^*_x|\gamma_x) < \infty$. Fix $(x,\pi^*_x),(z,\pi^*_z) \in \Gamma$. 
	
	We want to show that $\pi^*_{x} \sim \pi^*_{z}$. By $C$-monotonicity, cf.  Definition \ref{def:C-monotonicity}, we have for all $q_x,q_z \in \mathcal P(Y)$ with $q_x + q_z = \pi^*_{x} + \pi^*_{z}$ that
\begin{align*}
	H(\pi^*_{x}|\gamma_{x}) + H(\pi^*_{z}|\gamma_{z}) \leq H(q_x|\gamma_{x}) + H(q_z|\gamma_{z}).
\end{align*}
Then the corresponding first order optimality condition reads as 
\begin{align}\label{eq:rel ent first order}
	\pi^*_{x}\left(\log\left( \frac{d \pi^*_{x}}{d \gamma_{x}} \right)\right) + \pi^*_{z}\left(\log\left( \frac{d \pi^*_{z}}{d \gamma_{z}} \right)\right) \leq q_x\left(\log\left( \frac{d \pi^*_{x}}{d \gamma_{x}} \right)\right) + q_z\left(\log\left( \frac{d \pi^*_{z}}{d \gamma_{z}} \right)\right).
\end{align}
This inequality can be easily deduced by differentiation along the segment joining $\pi^*_{x}$ with $q_x$ and $\pi^*_{z}$ with $q_z$, i.e., by computing
$$\left. \frac{d}{dt}\right|_{t=0} H(q_x^t|\gamma_x) + H(q_z^t|\gamma_z),$$
where $q_\cdot^t = (1-t) \pi^*_{\cdot} + t q_\cdot$.
Assume that there exists a Borel measurable set $A\subset Y$ with $\pi^*_{x}(A) > 0$ but $\pi^*_{z}(A) = 0$. Then on $A$ we have for $\nu$-a.e.\ $y$ that $\frac{d \pi^*_{z}}{d \gamma_{z}}(y) = 0$. It is straightforward that one can find $q_x,q_z$ such that $q_x + q_z = \pi^*_{x} + \pi^*_{z}$, $q_z(A)>0$ and $q_x\left(\log\left( \frac{d \pi^*_{x}}{d \gamma_{x}} \right)\right)<+\infty$, so we omit the technical details. As a consequence we find
\begin{align*}
	q_x\left(\log\left( \frac{d \pi^*_{x}}{d \gamma_{x}} \right)\right) + q_z\left(\log\left( \frac{d \pi^*_{z}}{d \gamma_{z}} \right)\right) = -\infty,
\end{align*}
which contradicts the first order optimality criterion since the left-hand side of \eqref{eq:rel ent first order} is non-negative. Hence by symmetry $\pi^*_{x}$ is equivalent to $\pi^*_{z}$ and as a consequence we have $\pi^*_x \sim \nu$ for all $(x,\pi^*_x)\in\Gamma$.

Applying Lemma~\ref{lem:relent contradiction} below to the pair $\pi^*_{x},\pi^*_{z}$ yields a contradiction to $C$-monotonicity.
\end{proof}

\begin{lemma}\label{lem:relent contradiction}
	Let $p_i$, $\gamma_i$, $i\in \{1,2\}$ and $\nu$ be equivalent probability measures on $Y$. If $H(p_i|\gamma_i),~i=1,2,$ is finite and there exists no constant $\alpha \in \R$ such that
	\begin{align*}
		\frac{d p_1}{d \gamma_1} = \alpha \frac{d p_2}{d \gamma_2} \quad \nu\text{-a.e.},
	\end{align*}
	then there are two probability measures $q_1,q_2 \in \mathcal P(Y)$ with $q_1 + q_2 = p_1 + p_2$ and
	\begin{align*}
		H(p_1|\gamma_1) + H(p_2|\gamma_2) > H(q_1|\gamma_1) + H(q_2|\gamma_2).
	\end{align*}
\end{lemma}

\begin{proof}
	Since any Polish space is Borel-isomorphic to a measurable subset of $[0,1]$ we may assume that $Y = [0,1]$. By the inverse transform sampling we may assume that $\nu$ is the uniform distribution on $[0,1]$. For $i \in \{1,2\}$, define the densities
	$$g_i\colon [0,1] \to (0,+\infty),\quad g_ i = \frac{d p_i}{d \gamma_i};\quad f_i\colon [0,1] \to (0,+\infty),\quad f_i = \frac{d \gamma_i}{d \nu}.$$	
	By Lusin's theorem we find for any $\epsilon > 0$ a compact set $K_\epsilon\subset [0,1]$ with mass $\nu(K_\epsilon) \geq 1-\epsilon$ such that $h := \frac{g_1}{g_2}\colon [0,1] \to (0,+\infty)$ is continuous on $K_\epsilon$. By the assumption that $g_1$ is not a multiple of $g_2$ we find two disjoint intervals $[a_1,b_1],[a_2,b_2]\subset (0,+\infty)$, $b_1 < a_2$ with
	$$\nu(h^{-1}([a_1,b_1])) > 0 \text{ and } \nu(h^{-1}([a_2,b_2]))>0.$$
	Choosing $\epsilon > 0$ sufficiently small the closed, disjoint sets
	$$A_1 = h^{-1}([a_1,b_1]) \cap K_\epsilon \text{ and }A_2 = h^{-1}([a_2,b_2])\cap K_\epsilon$$
	have positive mass under $\nu$. W.l.o.g. we can assume that $\nu(A_1) = \nu(A_2) > 0$. The map
	\begin{align*}T&\colon A_1  \to A_2\\ &\hphantom{\colon} x \mapsto F_{\nu|_{A_2}}^{-1} \circ F_{\nu|_{A_1}}(x)
	\end{align*}
	provides a measure preserving bijection. Let $S\colon [0,1] \to [0,+\infty)$ be defined as
	%\comment{\color{orange} Die Gleichung war zu lang. Hab es deswegen bisschen umgeschrieben}
\begin{align*} 
	S(y) = \begin{cases} \min\Big\{g_1(y)f_1(y),\,g_2(y)f_2(y),\,\hat g_1(y)
\hat f_1(y),\,\hat g_2(y)\hat f_2(y)\Big \} & y \in A_1,\\
-\min\Big\{g_1(z)f_1(z),\,g_2(z)f_2(z),\,g_1(y)
f_1(y),\,g_2(y)f_2(y)\Big\} & y \in A_2,\\
0&\text{else,}
\end{cases}
	\end{align*}
	where $\hat g_i = g_i \circ T$, $\hat f_i = f_i \circ T$ and $z = T^{-1}(y)$.
Then we define for $t \in [-1,1]$ the probability measures $p_1^t$ and $p_2^t$ on $[0,1]$ by
\begin{align*}
	p_1^t(dy) = p_1(dy) + tS(y)\nu(dy),\quad p_2^t(dy) = p_2(dy) - tS(y) \nu(dy).
\end{align*}
The respective densities are then given by 
\begin{align*}
	\frac{d p_1^t}{d \gamma_1} = g_1 + t \frac{S}{f_1},\quad \frac{d p_2^t}{d \gamma_2} = g_2 - t \frac{S}{f_2}.
\end{align*}
Differentiation of
\begin{align*}
	& H(p_1^t|\gamma_1) + H(p_2^t|\gamma_2) \\= & \int_Y \log\left(g_1(y) + t \frac{S}{f_1}(y)\right) p_1^t(dy) + \int_Y \log\left(g_2(y) + t \frac{S}{f_2}(y)\right) p_2^t(dy)
\end{align*}
with respect to $t$ and evaluation at 0 yields
\begin{align*}
	\int_Y  S(y) \log\left( \frac{g_1}{g_2}(y) \right) \nu(dy) = \int_{A_1} S(y) \log\left(\frac{g_1(y) g_2(T(y))}{g_2(y) g_1(T(y))} \right)\nu(dy) < 0.
\end{align*}
By strict convexity of $t\mapsto H(p_1^t|\gamma_1) + H(p_2^t|\gamma_2)$ we conclude that there exists a $t_0 \in [-1,1]$ with
\begin{align*}
	H(p_1|\gamma_1) + H(p_2|\gamma_2) > H(p_1^{t_0}|\gamma_1) + H(p_2^{t_0}|\gamma_2).
\end{align*}

\end{proof}

Finally we obtain the main result of this section (see  \cite[Corollary 3.2]{Cs75}).
\begin{corollary}\label{Schroedinger_Char} Let $\mu \in \mathcal P(X)$ and $\nu \in \mathcal P(Y)$ and 
 $\gamma\in \mathcal P(X\times Y)$ a probability measure equivalent to $\mu\otimes\nu$. Assume that the  value of the corresponding entropic optimal transport problem is finite, i.e.,
	\begin{align}\label{eq:rel ent problem}
		\inf_{\pi \in \Pi(\mu,\nu)} H(\pi|\gamma) < \infty.
	\end{align}		
	A coupling $\pi^* \in \Pi(\mu,\nu)$ minimizes \eqref{eq:rel ent problem} if and only {\color{black}$H(\pi^*|\gamma)<\infty$ and} there exist measurable functions $f$ and $g$ %with $\log(f) \in L^1(\mu)$ and $\log(g) \in L^1(\nu)$ 
	such that
	\begin{align*}
		\frac{d \pi^*}{d \gamma}(x,y) = f(x) g(y)\quad \gamma\text{-a.e.}
	\end{align*}
\end{corollary}

\begin{proof}
%{\color{orange} 
%	\begin{align*} 
%		+\infty > H(\pi|\gamma) &= \int \log(\frac{d\pi_x d\mu}{d\gamma_x d\gamma_0})d\pi = \int \log(\frac{d\mu}{d\gamma_0} g h)d\gamma \\ &= \int \log(gh) gh \frac{d\mu}{d\gamma_0} + \log(\frac{d\mu}{d\gamma_0}) gh\frac{d\mu}{d\gamma_0} d\gamma
%	\end{align*}
%	Note that $x\log(x)$ is bounded from below and the measures $\frac{d\mu}{d\gamma_0}d\gamma$ and $gh d\gamma$ are probability measures. Thus,
%	\begin{align*}
%	\int \log(gh) gh \frac{d\mu}{d\gamma_0} + \log(\frac{d\mu}{d\gamma_0}) gh\frac{d\mu}{d\gamma_0} d\gamma &= \int \log(gh) gh \frac{d\mu}{d\gamma_0} + \int\log(\frac{d\mu}{d\gamma_0})\frac{d\mu}{d\gamma_0} gh d\gamma 
%	\end{align*}
%	We deduce that
%}
	We use the notation of Theorem~\ref{thm:rel ent} and rewrite
	\begin{align}\label{eq:rewrite rel ent}
		H(\pi|\gamma) = \int_X H(\pi_x|\gamma_x) \mu(dx) + H(\mu|\gamma_0),
	\end{align}
	where $\gamma_0 \in \mathcal P(X)$ is the $X$-marginal of $\gamma$. Lower semicontinuity and strict convexity of the relative entropy yield the existence of a minimizer $\pi^* \in \Pi(\mu,\nu)$ of \eqref{eq:rel ent problem}. The identity \eqref{eq:rewrite rel ent} shows that $\pi^*$ also minimizes \eqref{eq:weak rel ent problem}. By Theorem~\ref{thm:rel ent}, {calling $\Gamma$ the set with $\mu$-full $X$-projection therein}, we fix $(z,\pi^*_z)\in \Gamma$ and define $g(y):= \frac{d \pi^*_z}{d \gamma_z}(y)$ as well as $h(x)=\alpha$  for $\alpha$ as in \eqref{eq_pix_piz}. Thus we have
	\begin{align*}
		\frac{d \pi^*_x}{d \gamma_x}(y) = h(x) g(y)\quad (x,\pi^*_x)\in\Gamma,\, \gamma_x\text{-a.e. }y.
	\end{align*}
	Hence,
	\begin{align*}
		\frac{d \pi^*}{d \gamma}(x,y) = \frac{d \mu}{d \gamma_0}(x) h(x) g(y) = f(x) g(y)\quad \gamma\text{-a.e.},
	\end{align*}
	where $f$ is defined appropriately. Conversely, %suppose $f$ and $g$ are given as above and further assume that $\log(f) \in L^1(\mu)$ and $\log(g) \in L^1(\nu)$. By a well-known line of argument (where $\pi \in \Pi(\mu,\nu)$ with $\pi\ll \gamma$) we see
	%\begin{align*}
%		H(\pi|\gamma) - H(\pi|\pi^*) &= \int_{X\times Y} %\log(f(x)g(y))\pi(dx,dy)= \\
%		&= \int \log(f)d\mu +\int\log(g)d\nu= H(\pi^*|\gamma)
%	\end{align*}
%	that $\pi^*$ is optimal. In case $\log(f) \in L^1(\mu)$ resp.\ $%\log(g) \in L^1(\nu)$ fails, the result would follow by considering $-%n\vee \log(f) \wedge n $ resp.\ $-n\vee \log(g) \wedge n $ and %standard approximation arguments.
	{\color{black}assume that we have two couplings $\pi$, $\pi'$ {\color{black}with finite entropy,} with marginals $\mu$ and $\nu$, and such that their densities w.r.t.\ $\gamma$ are of product form. Let
	$$\frac{d\pi}{d\gamma}(x,y) = f(x) g(y)\text{ and }\frac{d\pi'}{d\gamma}(x,y) = f'(x)g'(y)$$
	and $\log(fg)fg, \log(f'g')f'g' \in L^1(\gamma)$. Since the marginals of $\pi$ and $\pi'$ coincide, we have for any $h \in L^1(\mu) \oplus L^1(\nu)$ that
	\begin{align}\label{eq:ortho}
		\int h(x,y) (fg - f'g') \gamma(dx,dy) = 0.
	\end{align}
	We can approximate $\log(fg)$ by elements in $L^1(\mu) \oplus L^1(\nu)$ such that on $[\log(fg) \geq 0]$ we have $h_n \geq 0$ and $h_n \nearrow \log(fg)$ and similar applies on $[\log(fg) \leq 0]$. Hence
	$$\int \log(fg)d\pi= \lim_n \int h_n fgd\gamma = \lim_n \int h_n f'g' d\gamma = \int \log(fg)d\pi'.$$
	Therefore, $\log(fg)f'g' \in L^1(\gamma)$ and $\log(f'g')fg \in L^1(\gamma)$ and
	\begin{align*}
		H(\pi|\gamma) - H(\pi|\pi') = \int \log(f'g') d\pi = \int \log(f'g') d\pi' = H(\pi'|\gamma).
	\end{align*}
	Especially, we have shown $H(\pi|\gamma) = H(\pi'|\gamma)$ since $H(\pi|\pi')$ is non-negative, which implies $H(\pi|\pi')$ and $H(\pi'|\pi)$ vanish, so $\pi = \pi'$.
	}
\end{proof}

\section{Brenier-Strassen Theorem}\label{se:Brenier--Strassen}

A fundamental result in the theory of optimal transport is  Brenier's theorem \cite{Br87, Br91} which asserts that the optimizer of the Wasserstein-2 distance on $\R^d$ between $\mu, \nu$ is given by the gradient of a convex function $\phi$. Specifically,  the optimal plan $\pi^* \in \Pi(\mu,\nu)$ is of the form $(id,\nabla \phi)(\mu)$ and $\nabla\phi(\mu) = \nu$. We refer to, e.g.\ \cite[Theorem 2.12]{Vi03} for more details and bibliographical remarks.

Strassen's theorem \cite{St65} asserts that given marginals $\mu, \nu\in \mathcal P_1(\R^d)$ there exists a martingale $(Z_i)_{i=1,2}$ with $Z_1\sim \mu, Z_2 \sim \nu$ provided that $\mu\leq_c \nu$. Here $\leq_c$ denotes the usual convex order, i.e.\ $\mu\leq_c \nu$ means that $\int\phi\, d\mu \leq_c \int \psi\, d\nu$ for all convex functions $\phi$. Note that the condition $\mu\leq_c\nu$ is not only sufficient for the existence of a martingale with these marginals but in fact also necessary by Jensen's inequality. 

An intermediate version between the classical $\mathcal W_2$-problem and Strassen's Theorem on the existence of martingales \cite{St65} was recently investigated in various forms \cite{GoJu18, GoRoSaTe17,GoRoSaSh18,AlCoJo17,Sh18,BaBePa18,BaBePa19}.
The following represents the weak transport analogue of the classical $\mathcal W_2$-distance:
		\begin{align}\label{eq:Brenier-Strassen} %\label{FirstBaryCost}
			V_2(\mu,\nu)^2 & := \inf_{\eta \leq_c \nu} \mathcal W_2(\mu,\eta)^2 \\ 
			&= \inf_{\pi\in \Pi(\mu,\nu)} \int_{\R^d} \left|x - \int_{\R^d} y \pi_x(dy)\right|^2 \mu(dx).
			\label{eq:Brenier-Strassen2} 
\end{align} 
Note that the equality of 
\eqref{eq:Brenier-Strassen} and \eqref{eq:Brenier-Strassen2} is a straightforward consequence of Strassen's Theorem. 	
		
		The Brenier-Strassen theorem of Gozlan and Juillet \cite{GoJu18} (see also Shu \cite{Sh16} for the one-dimensional case) asserts that $\eta^* \leq_c \nu$ is optimal for \eqref{eq:Brenier-Strassen} if and only if there exists a convex function $\phi$ with 1-Lipschitz gradient such that $\nabla \phi(\mu) = \eta^*$. 			The proof of the \emph{if}-part of the statement can be done by showing dual attainment or more directly using $C$-monotonicity whereas the \emph{only if}-part was thus far only shown via duality. 
		
		The goal of this section is to recover the \emph{only if}-clause by means of sufficiency of $C$-monotonicity. We believe that this new proof is appealing in that it mimics the proof of Brenier's theorem via classical cyclical monotonicity, underlining  the similarity of the two results.  %, cf.  \cite[Theorem 2.2]{BaPa19}. 

				To provide  an intuition for the proof, we recall the main idea in the classical case:	
		Let $\phi\colon\R^d \to \R$ be a convex and differentiable function. For any finite number of points $x_1,\ldots,x_n \in\R^d$ it is immediate that
		\begin{align}\label{eq:cycl mon}
			\sum_{k=1}^n \langle x_{i+1} - x_i, \nabla \phi(x_i) \rangle \leq 0.
		\end{align}
		Completing the square yields
		\begin{align}\label{eq:compl square}
			\sum_{k=1}^n |x_i-\nabla \phi(x_i)|^2 \leq \sum_{k=1}^n |x_{i+1} - \nabla \phi(x_i)|^2,
		\end{align}
		which shows optimality for any measure supported on finitely many points in the graph of $\nabla \phi$. Therefore, since cyclical monotonicity is sufficient for  optimality,  %(or continuity of the optimal transport problem wrt \ its marginals like Theorem \ref{thm:existence}), 
		the coupling $(id,\nabla \phi)(\mu)$ turns out to be optimal for the quadratic distance transport problem. 
		
		In the subsequent proof, we will use as a black box a characterization of convex functions with 1-Lipschitz gradient in the spirit of cyclical monotonicity, see \cite{Zh18}.  This permits to draw the desired  conclusion for the problem \eqref{eq:Brenier-Strassen} in analogy to step from  \eqref{eq:cycl mon} to \eqref{eq:compl square}.

	%\comment{MB: Sollten wir hier nicht noch das Brenier-Strassen Theorem selbst angeben und beweisen? Vielleicht als Korollar? Ich glaube das ist nicht allen Lesern sofort klar...}
	\begin{theorem}\label{Brenier-Strassen}%\comment{\color{orange} $\eta$ zu $\eta^*$ und $(\cdot)^2$ zu $|\cdot|^2$}
		Let $\mu, \nu \in \mathcal P_2(\R^d)$ where $\R^d$ is equipped with the standard euclidean norm $|\cdot|$. Let $\eta^*$ be in convex order dominated by $\nu$. Then the following are equivalent:
		\begin{enumerate}
			\item \label{it:bs1} The values of $\mathcal W_2(\mu,\eta^*)$ and $V_2(\mu,\nu)$ coincide, i.e.\ $\eta^*$ is a solution of the optimization problem \eqref{eq:Brenier-Strassen}
			\item \label{it:bs2} There is a  convex function $\phi\colon \R^d\to\R$ such that $\nabla \phi\colon\R^d\to\R^d$ is 1-Lipschitz, $\eta^* = \nabla\phi(\mu)$ and $(id,\nabla\phi)_{\#}\mu$ is the unique optimizer of \eqref{eq:Brenier-Strassen}.
		\end{enumerate}
	\end{theorem}
Theorem \ref{Brenier-Strassen} can be found in  \cite[Theorem 2.1]{GoJu18}. A proof of the first implication using different arguments is given in \cite[Theorem 1.4]{BaBePa18}.
	
	\begin{proof}[Proof of Theorem \ref{Brenier-Strassen}]
		Here we will only show ``\ref{it:bs2}$\implies$ \ref{it:bs1}".
		To this end, we want to verify that $$\Gamma := \left\{(x,\delta_{\nabla \phi(x)}) \in \R^d\times \mathcal P(\R^d)\colon x \in \R^d \right\}$$
		is $C$-monotone, where $C(x,p):= \big|x-\int yp(dy)\big|^2$. 
		
		Let $N\in\N$, $x_1,\ldots,x_N \in \R^d$ with $y_i = \nabla\phi(x_i)$.
 We have to show that if $\sum_{i=1}^N m_i =  \sum_{i=1}^N\delta_{y_i}$ then $\sum_{i=1}^N |x_i-y_i|^2\leq \sum_{i=1}^N |x_i-\int ym_i(dy)|^2$. Clearly we must have $m_i = \sum_{j=1}^N \alpha_{i,j} \delta_{y_j} $ for all $i$, where $(\alpha_{i,j}) \in \R^{N\times N}_+$ is a bistochastic matrix. 
%Let us denote $z_1,\ldots,z_N\in\R^d$ the barycenters of $m_1,\dots,m_N$, namely}			%		 By convexity of the square it is sufficient to only consider competitors which consist of Dirac measures on $N$ points. These are denoted here by $z_1,\ldots,z_N\in\R^d$ with
%		\begin{align*}
%			z_i = \sum_{j=1}^N \alpha_{i,j} y_j\quad i \in \{1,\ldots,N\},
%		\end{align*}
		We can rewrite $\alpha = (\alpha_{i,j})$ as a convex combination of permutation matrices $(P_{\sigma})_{\sigma \in \Sigma}$, where $\Sigma$ denotes the set of permutations on $\{1,\ldots,N\}$:
		\begin{align*}
			\alpha = \sum_{\sigma\in\Sigma} \beta_\sigma P_\sigma,\quad \sum_{\sigma\in\Sigma} \beta_\sigma = 1,\quad \beta_\sigma \geq 0.
		\end{align*}
		Note that the map $F\colon\R^{N!} \to \R$
		\begin{align*}
			F((\beta_\sigma)_{\sigma\in\Sigma}) =\frac{1}{2} \sum_{i=1}^N\Big | x_i - \sum_{\sigma\in\Sigma} \beta_\sigma y_{\sigma(i)} \Big |^2
		\end{align*}
		is convex. We seek to show that
		\begin{align*}
			\hat \beta_\sigma = \begin{cases} 1 & \sigma = id,\\ 0 & \text{else,} \end{cases}
		\end{align*}
		is a minimum on the simplex on $\R^{N!}$. The gradient of $F$ at $\hat \beta$ has the following form
		\begin{align*}
			\nabla F(\hat \beta) = \left(\sum_{i=1}^N (x_i - y_i)\cdot (y_i - y_{\sigma(i)})\right)_{\sigma\in\Sigma}.
		\end{align*}
		For any permutation $\sigma\in\Sigma$ we compute
		\begin{align*}
			\sum_{i=1}^N (x_i - y_i) \cdot (y_i - y_{\sigma(i)}) &= \sum_{i=1}^N x_i \cdot (y_i - y_{\sigma(i)}) - \frac{1}{2} | y_i - y_{\sigma(i)}|^2 \\
			&= \sum_{i=1}^N (x_i - x_{\sigma^{-1}(i)})\cdot y_i - \frac{1}{2} | y_i - y_{\sigma^{-1}(i)}|^2  \\ &\geq 0.
		\end{align*}
	The last inequality above follows from \cite[Lemma 4]{Zh18}, which states that $\phi$ being convex and $1$-Lipschitz is equivalent to
		\begin{align*}
			\phi(z) - \phi(x) \geq \nabla \phi(x) \cdot (z - x) + \frac{1}{2} | \nabla \phi(z) - \nabla \phi(x)|^2\quad \forall x,z\in\R^d,
		\end{align*}
		so by summing this last inequality we get
		\begin{align*}
		 0 &= \sum_{i=1}^N \phi(x_{\sigma(i)}) - \phi(x_i) \\
		 &\leq \sum_{i=1}^N (x_i - x_{\sigma(i)}) \cdot y_i - \frac{1}{2} | y_i -  y_{\sigma(i)}|^2.
		\end{align*}
		We conclude that $\nabla F(\hat \beta)$ is pointwise non-negative proving optimality of $\hat\beta$ on the simplex by convexity of $F$.
		
	All in all, the coupling $(id,\nabla\phi)_\# \mu$ is concentrated on the $C$-monotone set $\Gamma$. Since both marginals are in $\mathcal P_2(\R^d)$ we can apply Theorem \ref{thm:stability} which yields its optimality for the weak transport problem $V_2(\mu,\nu)$.
	\end{proof}
	
	An immediate consequence of the  Brenier-Strassen Theorem \ref{Brenier-Strassen}, coupled with the proof therein, is the following modification of the classical Rockafellar theorem. We denote the subdifferential of a convex function $\phi$ at $x \in \R^d$ by $\partial \phi(x)$ and by $\partial \phi = \{(x,y)\colon x \in \R^d, y \in \partial \phi(x)\}$ its graph.
	
	\begin{corollary}[Rockafellar-Strassen]
		Let $L \in \R^+$ and $\Gamma \subset \R^d \times \R^d$. The following are equivalent:
		\begin{enumerate}
			\item \label{it:rs1} $\Gamma$ satisfies for all\footnote{Here $(x_{n+1},y_{n+1}):=(x_1,y_1)$.} $n \in \N,\, ((x_1,y_1),\ldots,(x_n,y_n)) \in \Gamma$,
		\begin{align}\label{eq:strassenmonotone}
			\sum_{i=1}^n (x_{i+1} - x_i) \cdot y_i + \frac{1}{2L}|y_{i+1}-y_i|^2 \leq 0.
		\end{align}
			\item \label{it:rs2} there exists a convex $\phi \in C^1(\R^d)$ with L-Lipschitz gradient such that $\Gamma \subset \partial \phi$.

		\end{enumerate}
	\end{corollary}
	
	\begin{proof}
		The implication `\ref{it:rs2}$\implies$\ref{it:rs1}' can be easily deduced from \cite[Lemma 4]{Zh18}.
	For `\ref{it:rs1}$\implies$\ref{it:rs2}': Since we can always consider $\tilde \Gamma = \{(x,y) \colon (x,Ly) \in \Gamma\}$, which sastisfies \eqref{eq:strassenmonotone} for $L=1$, we can assume w.l.o.g.\ that $L = 1$. Let $\mu$ be a probability measure on $\R^d$ such that
		$$\supp(\mu) = \text{cl}(\{x \colon \exists y,(x,y) \in \Gamma \}).$$
		By \eqref{eq:strassenmonotone}, if $(x,y), (x,y') \in \Gamma$ then $y = y'$. This defines a map $T\colon \proj_1 \Gamma \to \proj_2 \Gamma$. Further, for all $(x,y),(x',y') \in \Gamma$, we have
		$$-|y-y'|^2 \geq (x-x')\cdot (y-y') \geq -|x-x'| |y-y'|,$$
		which shows that $T$ is 1-Lipschitz. The coupling $\pi = (id,T)_\# \mu$ is optimal for \eqref{eq:Brenier-Strassen2} between its marginals $\mu$ and $\eta^*$, by the reasoning in the proof of Theorem \ref{Brenier-Strassen} . Hence, by Theorem \ref{Brenier-Strassen} there is a convex function $\phi$ with 1-Lipschitz gradient such that $(id,\nabla \phi)_\# \mu$ is the unique optimizer. Thus, $\nabla \phi = T$ ($\mu$-a.s.). Due to continuity we conclude $\nabla \phi = T$ on $\proj_1 \Gamma$ and $\Gamma \subset \partial \phi$.
	\end{proof}
	
	\section{Multiple-Good monopoly problem}\label{se:Daskalakis}
	
The goal of this section is to recover the main result of  Daskalakis, Deckelbaum, and Tzamos \cite{DaDeTz17} in Corollary~\ref{cor:daskalakis}  below. We will not discuss the economic interpretation and just mention that it can be interpreted as a Kantorovich-Rubinstein Theorem for specific weak transport costs. 

We will obtain Corollary~\ref{cor:daskalakis} as a consequence of the more general result Theorem \ref{thm:daskalakis}. We first introduce some notation. 
	Fix $X:=\R^d$ and equip it with the coordinate-wise partial order $\leq$. Let $\Phi_t^{icx}(\R^d) $ consist of all $\leq$-increasing, convex functions in $\Phi_t(\R^d)$. For $\mu,\nu \in \mathcal P_1(X)$
	we write $\mu \leq_{icx} \nu$ iff $\int f\, d\mu \leq \int f\, d\nu$ for all $f \in \Phi_1^{icx}(X)$. Once again, it follows from the results of Strassen \cite{St65} that 
$\mu \leq_{icx} \nu$ is tantamount to the existence of a stochastic process $(Z_i)_{i=1,2}$ satisfying %\comment{\color{orange} $Z_0$, $Z_1$ zu $Z_1$, $Z_2$}
$$ Z_1 \sim \mu, Z_2\sim \nu, Z_1 \leq \E[Z_2|Z_1].$$
	
	%The set $\Phi_t^{icx}(\R^d)$ satisfies \ref{it:D1}-\ref{it:D4} and induces the partial order $\leq_{icx}:=\leq_{\Phi_t^{icx}(\R^d)}$ on $\mathcal P_t(\R^d)$. That is, $\leq_{icx}$ is the increasing convex order.
	\begin{theorem}\label{thm:daskalakis}
		Let $\theta \in \Phi_{b,t}(\R^d)$ be convex and denote $$\textstyle C_{\theta,icx}(x,p) := \inf_{q \leq_{icx} p} \theta\left(x-\int y q(dy)\right)\text{ and }R_\theta \phi(x) := \inf_{y \leq z,z\in\R^d} \phi(z) + \theta(x-y).$$
		Then all of the following optimization problems yield the same value:
		\begin{enumerate}[label = \roman*)]
		\item \label{it:dask1}$\inf_{\pi \in \Pi(\mu,\nu)} \int_{\R^d} C_{\theta,icx}(x,\pi_x) \mu(dx)$,
		\item \label{it:dask2}$\inf_{\tilde \nu \leq_{icx} \nu} \inf_{\pi \in \Pi(\mu,\tilde \nu)}\int \theta(x-y) \pi(dx,dy)$,
		\item \label{it:dask3}$\inf_{\mu \leq_{icx} \tilde \mu, \tilde \nu \leq_{icx} \nu} \inf_{\pi \in \Pi(\tilde \mu,\tilde \nu)}\int \theta(x-y) \pi(dx,dy)$,
		\item \label{it:dask4}$\sup_{\phi \in \Phi_t^{icx}(\R^d)} -\nu(\phi) + \int R_\theta \phi(x)\mu(dx)$.
		\end{enumerate}
	\end{theorem}
	
	\begin{proof}
		``$i)=ii)$": For $(i)\geq (ii)$ we first take $\pi$ to be an (almost) optimizer of $(i)$ and by a measurable selection argument take $q_x$ to be  an (almost) optimizer of $C_{\theta,icx}(x,\pi_x)$. Defining $T(x):=\int y q_x(dy)$ we remark that $T(\mu)\leq_{icx} \nu$, and so $(i)\geq \int \theta(x-T(x))\mu(dx)\geq (ii)$. For the converse, take $\tilde \nu \leq_{icx} \nu$ and $P\in \Pi(\mu,\tilde \nu)$ which are (almost) optimizers of $(ii)$. By Strassen's result \cite[Theorem 9]{St65} 
		 there exists a submartingale coupling $\tilde\pi \in \Pi(\mu,\nu)$, i.e.\ a coupling $\pi $ satisfying $\int y\, d\pi_x \geq x$. 
		 
		 Next we define $\pi(dx,dy)=\int\tilde\pi_{\tilde x}(dy) P(dx,d\tilde x)$ which belongs to $\Pi(\mu,\nu)$. Since by definition $\theta(x-\tilde x)\geq C_{\theta,icx}(x,\tilde \pi_{\tilde x})$, we have by Jensen's inequality
		\begin{align*}(ii)=\int \theta(x-\tilde x)P(dx,d\tilde x)&\geq \int C_{\theta,icx}(x,\tilde \pi_{\tilde x})P(dx,d\tilde x) \\ &\geq \int C_{\theta,icx}\Big (x,\int\tilde \pi_{\tilde x}P_{x}(d\tilde x) \Big)\mu(dx)\\
&\geq  \int C_{\theta,icx} (x,\pi_x )\mu(dx)		 \geq (i),
		 \end{align*} 
where we used that $C_{\theta,icx}(x,\cdot)$ is convex if $\theta$ is convex.		
		
%		Is a consequence of Jensen's inequality, Theorem~\ref{thm:strassen} and the measurable selection theorem~\cite[Proposition 7.50]{BeSh78}.
				``$ii) = iii)$": Fix $\mu \leq_{icx} \tilde \mu$ and $\tilde \nu \leq_{icx} \nu$. Let $X_1 \leq_{icx} X_2$ and $Y_1 \leq_{icx} Y_2$ be $\R^d$-valued random variables on some probability space whose laws satisfy
		$$X_1 \sim \mu,\quad X_2 \sim \tilde \mu,\quad Y_1 \sim \tilde \nu, \quad Y_2 \sim \nu.$$
		The order $\leq_{icx}$ between random variables has to be understood in the following sense, (where the existence of such random variables is again provided by \cite[Theorem 9]{St65})
		\begin{align*}
			X \leq_{icx} Y \iff \mathbb E[Y|X] \geq X\quad \text{a.s.}
		\end{align*}
		Define the random variable $Z = X_1 + \E[Y_1 - X_2|X_1]$. Then $Z$ is in increasing convex order to $Y_2$, since
		\begin{align*}
			Z \leq_{icx} Z + \E[X_2 - X_1|X_1] = \E[Y_1|X_1] \leq_{icx} Y_1 \leq_{icx} Y_2.
		\end{align*}
		 Applying Jensen's inequality yields
		\begin{align*}
			\E[\theta(X_2-Y_1)] \geq \E[\theta(\E[X_2-Y_1|X_1])] = \E[\theta(X_1 - Z)] \geq ii).
		\end{align*}
		
		``$i) = iv)$": By duality, see Theorem \ref{thm:duality}, we have
		\begin{align}\label{eq:application KD}
			i) = \sup_{\phi \in \Phi_{b,t}(\R^d)} -\nu(\phi) + \int R_{C_{\theta,icx}} \phi(x) \mu(dx).
		\end{align}
		It remains to show that we can additionally restrict the infimum  to convex and increasing functions which will be accomplished using the double $c$-convexification trick. To this end, we note that 
		\begin{align*}\textstyle
			C_{\theta,icx}(x,p) = \inf_{x - \int y p(dy) \leq z} \theta(z) =: \hat \theta\left(x-\int yp(dy)\right),
		\end{align*}
		where $\hat \theta$ is a function on $\R^d$ which is bounded from below and convex. Further note
		\begin{align}\label{eq:inc dec}
			x \mapsto \hat \theta(x-y) \text{ is increasing},\quad y \mapsto \hat\theta(x-y)\text{ is decreasing}.
		\end{align}
		Analogously to Gozlan et al.\ \cite[Proof of Theorem 2.11 (2)]{GoRoSaSh18} we find that
		\begin{align*}
			\inf_{y\in\R^d} \phi(y) + \hat\theta(x-y) = R_{C_{\theta,icx}} \phi(x) = R_{C_{\theta,icx}} \hat \phi(x) = \inf_{y \in \R^d} \hat\phi(y) + \hat \theta(x-y),
		\end{align*}
		where $\hat\phi$ denotes the convex envelope of $\phi$. The $\inf$-convolution $\psi := R_{C_{\theta,icx}}\hat \phi$ is therefore bounded from below, convex and increasing. Note that for any $p \in \mathcal P(\R^d)$ with barycenter $\int y p(dy) = z$ we have
		\begin{align*}
			\psi(x) - p(\hat \phi) \leq \psi(x) - \hat \phi(z) \leq \hat \theta(x-z),
		\end{align*}
		which allows us to $\hat \theta$-convexify $\hat \phi$, i.e.,
		\begin{align*}
			\bar \phi(y) = \sup_x \psi(x) - \hat \theta(x-y),
		\end{align*}
		which is in particular an increasing function by \eqref{eq:inc dec} with $$\hat\phi \geq \bar \phi \geq \psi(y) - \hat \theta(0)\geq \min_y \psi(y) - \hat \theta(0).$$ Again, we find that the convex envelope of $\bar \phi$
		\begin{align}\label{eq:def tilde phi}
			\tilde \phi(z) = \inf_{p\in\mathcal P_t(Y)\colon z = \int y p(dy)}  p(\bar\phi)
		\end{align}
		is increasing, convex and dominated by $\bar \phi$,  thus,
		\begin{align*}
			\tilde \psi(x) := R_{C_{\theta,icx}} \tilde \phi (x) = R_{C_{\theta,icx}} \bar \phi(x)  \geq \psi(x) = R_{C_{\theta,icx}} \phi(x).
		\end{align*}
		Finally, we obtain (since $\tilde \phi \leq \phi$ and $\psi \leq \tilde \psi$) that
		\begin{align*}
			\mu(\psi) - \nu(\phi) \leq \mu(\tilde \psi) - \nu(\tilde \phi),
		\end{align*}
		which shows that we can replace $\Phi_{b,t}(\R^d)$ with $\Phi_{b,t}^{icx}(\R^d)$ in \eqref{eq:application KD}.
	\end{proof}
	
	The proof of Corollary~\ref{cor:daskalakis} resembles the proof of Kantorovich-Rubinstein duality when one already knows that classical Kantorovich duality holds.	
		
	\begin{corollary}\label{cor:daskalakis} Setting  $\theta(x) = \lvert x \rvert$ where $|\cdot|$ is a norm on $\R^d$, we have
		\begin{align*}
			\inf_{\pi \in \Pi(\mu,\nu)} \int_{\R^d} \inf_{z \leq \int y \pi_x(dy)} |x-z| \mu(dx) = \sup_{\phi\in \Phi_{b,1}^{icx}(\R^d)\text{ and 1-Lipschitz}} \mu(\phi) - \nu(\phi).
		\end{align*}
	\end{corollary}

	\begin{proof}
		First we show that given $\phi \in \Phi_{b,1}^{icx}(\R^d)$ the $\inf$-convolution $\psi := R_{C_{|\cdot |,icx}}\phi$ is 1-Lipschitz: Let $x,x' \in \R^d$ then
		\begin{align*}
			R_{C_{|\cdot|,icx}} \phi(x) - R_{C_{|\cdot|,icx}} \phi(x') \leq \sup_{z\leq y,\ y \in \R^d} |x-y| - |x'-y| \leq |x-x'|.
		\end{align*}
		By the proof of Theorem~\ref{thm:daskalakis} $\psi$ is additionally bounded from below, convex and increasing. Using the notation in the proof of Theorem~\ref{thm:daskalakis} the mapping $\bar \phi$ is increasing, bounded from below and also 1-Lipschitz. Hence, for $x,x'\in\R^d$ the increasing, convex and lower bounded function $\tilde \phi$, see \eqref{eq:def tilde phi}, is 1-Lipschitz
		\begin{align*}
			\tilde \phi(x) - \tilde \phi(x') \leq \sup_{p \in \mathcal P_1(Y)} \mathcal W_1(p, T_{x'-x}(p)) = |x-x'|,
		\end{align*}
		where $T_{x'-x}$ is the translation by $x'-x$. By the conclusion of the proof of Theorem~\ref{thm:daskalakis} we can assume w.l.o.g.\ that $\phi \in \Phi_{b,1}^{icx}(\R^d)$ is 1-Lipschitz.
		
		It remains to show that $R_{C_{|\cdot|,icx}}\phi$ and $\phi$ coincide. By definition of the $\inf$-convolution we have $\phi \geq R_{C_{|\cdot|,icx}}$. We find by 1-Lipschitz continuity of $\phi$
		\begin{align*}
			- \phi(x) + \phi(y) + \inf_{z \leq y} |x-z| \geq \inf_{z \leq y} -\phi(x) + \phi(z) + |x-z| \geq 0,
		\end{align*}
		which shows the reverse inequality.		
	\end{proof}
	
%	{\color{red}Check}

\section{Backward transfers}\label{se:transfers}

Bowles and Ghoussoub \cite{BoGh19} suggest a notion of linear transfer between probability measures which is more encompassing than mass transportation but still admits important traits of the dual theory of mass transport. In particular, they identify many examples that illustrate the scope of their approach. 

A main result of Bowles and Ghoussoub yields a representation of linear transfers through weak transport problems, see  \cite[Theorem 3.1]{BoGh19}. In this section, we recover \cite[Theorem 3.1]{BoGh19} as an application of the weak transport duality theorem. Notably, this approach extends the result of Bowles and Ghoussoub from compact to general Polish spaces without additional effort.\footnote{As noted on \cite[page 3]{BoGh19}, the right setting for most applications of linear transfers should be `$\ldots$ complete metric spaces, Riemannian manifolds or at least $\R^n$.'. In this respect the extension of \cite[Theorem 3.1]{BoGh19} beyond the compact setup seems relevant.}

To present the notion of linear transfer, we introduce some notation. 
The basic object of interest are functionals $\mathcal T\colon \mathcal P(X)\times \mathcal P_t(Y)\to \R\cup \{+\infty\}$. We will use the Legendre transform $\mathcal T^*_\mu$ of $\mathcal T_\mu = \mathcal T(\mu,\cdot)$, which is given by
\begin{align*}
	\mathcal T^*_\mu(g) = \sup_{\nu \in \mathcal P_t(Y)} \nu(g) - \mathcal T(\mu,\nu).
\end{align*}
Since the set $\Phi_{t}(Y)$ is in separating duality with $\mathcal M_t(Y)$, the Fenchel duality theorem \cite[Theorem 2.3.3]{Za02} states  
\begin{align}\label{eq:fenchel duality}
	\mathcal T_\mu(\nu) = \mathcal{T}_\mu^{**}(\nu) = \sup_{g \in \Phi_t(Y)} \nu(g) - \mathcal T^*_\mu(g),
\end{align}
if $\mathcal T_\mu$ is proper convex, bounded from below and lower semicontinuous.

\begin{definition}\label{def:backward transfer}
	A proper, convex, bounded from below and lower semicontinuous functional $\mathcal T\colon \mathcal P(X) \times \mathcal P_t(Y)\to \R \cup \{+\infty\}$ is called backward linear transfer if there exists a map $T$ from $\Phi_t(Y)$ to the set of universally measurable functions on $X$ bounded from below by an element in $\Phi_t(X)$, with the following property: for each $\mu \in \mathcal P(X)$ with $\inf_{\nu \in \mathcal P_t(Y)}\mathcal T(\mu,\nu) < \infty$ the Legendre transform $\mathcal T_\mu^*$ can be represented as
	\begin{align}\label{eq_transfers_defi}
		\mathcal T_\mu^*(g) = \mu(T(g))\quad \forall g \in \Phi_t(Y).
	\end{align}
\end{definition}

\begin{theorem}
	Let $\mathcal T\colon \mathcal P(X) \times \mathcal P_t(Y)\to \R\cup\{+\infty\}$ be such that
	\begin{align}\label{eq:dirac in support}
		\forall x\in X,~ \exists p\in \mathcal P_t(Y)\colon\quad \mathcal T(\delta_x,p) < \infty.
	\end{align}
	Then the following are equivalent
	\begin{enumerate}[label = (\roman*)]
		\item \label{it:ghoussoub1}$\mathcal T$ is a backward linear transfer,
		\item \label{it:ghoussoub2} there is a lower semicontinuous cost function $C\colon X \times \mathcal P_t(Y)\to \R\cup \{+\infty\}$ which is bounded from below and convex in the second argument such that for all $(\mu,\nu) \in \mathcal P(X)\times \mathcal P_t(Y)$
		$$\mathcal T(\mu,\nu) = \inf_{\pi \in \Pi(\mu,\nu)} \int_X C(x,\pi_x)\mu(dx).$$
	\end{enumerate}
\end{theorem}

\begin{proof}
	Evidently, if $\mathcal T$ is given as a backward linear transfer, the cost function $C$ has to satisfy $C(x,p) = \mathcal T(\delta_x,p)$ and $T$ satisfies on $\Phi_t(Y)$
	$$T(g)(x) = \mathcal T^*_{\delta_x}g = \sup_{p \in \mathcal P_t(Y)}  p(g) - C(x,p) = -R_C(-g)(x).$$	
	Hence by \eqref{eq_transfers_defi} and \weakoptimal transport duality (Theorem~\ref{thm:duality}) we have
	\begin{align*}
		\mathcal T(\mu,\nu) &= \mathcal T_\mu(\nu) = \sup_{g \in \Phi_{t}(Y)} \nu(g) - \mathcal T_\mu^*(g) \\ &=  \sup_{g \in \Phi_{t}(Y)} \nu(g) - \int_X R_C(-g)(x) \mu(dx) \\ &= \sup_{g \in \Phi_t} -\nu(g) + \int_X R_C g(x) \mu(dx) \\ &= \inf_{\pi \in \Pi(\mu,\nu)} \int_X C(x,\pi_x)\mu(dx).
	\end{align*}
Conversely, if \ref{it:ghoussoub2} holds, then Theorem~\ref{thm:duality} reveals $$\mathcal T_\mu(\nu)= \sup_{g \in \Phi_{t}(Y)} \nu(g) - \mu( T(g)),$$
for $T(g)(x)=R_C(-g)(x)$, and by \eqref{eq:duality convex conjugate} we have $\mu(T(g))=\mathcal T_\mu^*(g)$.
\end{proof}	

\section{Semimartingale Transport Duality}
\label{se:TaTo}

In this part we need to set up some terminology before stating the actual problem. Let $$\C=C([0,1];\R^d)$$ denote the continuous path space equipped with the supremum norm $\|\cdot\|_\infty$ and its Borel $\sigma$-field. With $$W=(W(t))_{t \in [0,1]}$$ we denote the canonical (coordinate) process on $\C$, defined by $W(t)(\omega)=\omega(t)$, so that $W$ is a standard $d$-dimensional Brownian motion under the Wiener measure $\W$. Let $\FF=(\F_t)_{t \in [0,1]}$ denote the $\W$-complete filtration generated by $W$.
As usual, we denote by $L^0(\W)$ the space of (real-valued) random variables quotiented with the $\W$-a.s.\  identification, and by $L^\infty(\W)$ the essentially bounded elements of $L^0(\W)$. We will likewise identify processes that are $dt\times d\W$-almost surely equal. Finally, we denote by $\SS_+^d$ the set of symmetric positive semi-definite matrices of size $d\times d$. We fix from now on a matrix norm on $\R^{d\times d}$.\\

We consider $$g : [0,1] \times \R^d \times \R^d \times \SS_+^d  \rightarrow \R \cup\{\infty\}, $$ %a function whose effective domain we define as $\mathrm{dom}(g(t,x,q,\cdot)) := \{a \in \SS_+^d\, :\, g(t,x,q,a) < \infty \}$
 and assume 
 \begin{assumption}\label{ass:TT}
$ $

\begin{enumerate}
\item $g$ is jointly measurable and lower-bounded.
\item For each $t\in [0,1]$ the function $$\R^d \times\R^d \times \SS_+^d  \ni(q,a)\mapsto g(t,x,q,a),$$
is jointly lower semicontinuous. Furthermore $$(q,a)\mapsto g(t,x,q,a)$$ is convex for each fixed $(t,x)$.
\item Either $g$ is finite and coercive in the sense that
\begin{align}\lim_{{|q|\vee |a|}\to \infty}\,\inf_{t,x}\frac{g(t,x,q,a)}{|q|+|a|}=+\infty,\label{eq:coer_g}
\end{align}
or  $$\mathrm{dom}(g(t,x,\cdot,\cdot)):=  \{(q,a) \in \R^d\times\SS_+^d\, :\, g(t,x,q,a) < \infty \}$$ is a compact convex set which does not depend of $(t,x)$.
\end{enumerate}
\end{assumption}
\medskip

For $Q\in\mathcal P(\C)$ we denote by $$\M^{ac}_0(Q)$$ the space of continuous $\R^d$-valued $Q$-martingales, which are started at zero, whose quadratic variation matrix is absolutely continuous and integrable: Namely $M\in \M^{ac}_0(Q) $ iff it is a $Q$-martingale started at zero, $\frac{d\langle M \rangle_t }{dt}$ exists $Q$-a.s.\  and $$\E^Q[|\langle M \rangle(1)|]<\infty.$$ This last condition is equivalent to asking $$\E^Q\left [ \int_0^1 \left| \frac{d\langle M \rangle(t) }{dt}  \right |dt \right ]<\infty.$$ On the other hand we write $\L^1(Q)$ for the set of progressively measurable $\R^d$-valued processes which are integrable with respect to $dt\times dQ$.

 We can now introduce the set of semimartingale laws relevant to our work:
$$
\S:=\Big\{ Q\in \mathcal P(\C):\,
\begin{array}{c}
 W(\cdot)=\int_0^\cdot q^Q(s)ds + M^Q(\cdot)\text{ under $Q$}, \\
 \text{ for some $M^Q\in \M^{ac}_0(Q),\, q^Q \in \L^1(Q)$} 
\end{array} 
\Big\}
$$ 
%\begin{align*}
%\S:=\left\{ Q\in \mathcal P(\C):\, W(\cdot)=\int_0^\cdot q^Q(s)ds + M^Q(\cdot)\text{ under $Q$, $M^Q\in \M^{ac}_0(Q),\, q^Q \in \L^1(Q)$}  \right \}.
%\end{align*}
We remark that for $Q\in \S$ the process $q^Q$ above is uniquely determined. Likewise, the $\SS_+^d$-valued process $$a^Q:=\frac{d\langle M^Q\rangle_t}{dt}$$ is uniquely determined. \\

Let
\begin{align*}
\alpha^g :{\S}\to \R\cup\{+\infty\},
\end{align*}
be given by
\begin{equation}\label{eq def alpha rho}
	\alpha^g(Q):= \E^Q\left[\int_0^1g(t,W(t),q^Q(t),a^Q(t))\,dt\right].
\end{equation}
Note that $\alpha^g(Q)$ is well-defined and takes values in $ \R\cup\{+\infty\}$, as $g$ is bounded from below. As a final bit of notation, we introduce $\Pi(\mu,\nu)$ for the set of those $Q\in S$ with initial and final marginals equal to $\mu$ and $\nu$ respectively. We further write $\Pi(\mu,\cdot)$ when only the inital marginal is prescribed.

For $\mu,\nu\in\mathcal P(\mathbb R^d)$ we define\footnote{Equivalently, one may minimize the functional $\E[\int_0^1g(t,X(t),q(t),\sigma(t)\sigma'(t))dt]$ over all semimartingales $dX(t)=q(t) dt+\sigma(t) dB(t)$ on some stochastic basis, such that $X(0)\sim\mu,X(1)\sim\nu$. }
$$V(\mu,\nu):=\inf_{Q\in \Pi(\mu,\nu)}\alpha^g(Q).$$
This is a stochastic mass transport problem (or optimal transport of semimartingales) as introduced by Tan and Touzi \cite{TaTo13}. {Specification of $\alpha^g$ allow to cover classical optimal transport, martingale transport, and some instances of the Schr\"odinger problem in this framework.} We now prove a duality result, originally obtained by the aforementioned authors, by means of \weakoptimal transport:
\begin{theorem}
Under the standing assumptions, we have
$$V(\mu,\nu)=\sup_{\psi\in C_b(\mathbb R^d)}\left\{ \mu(\tilde\Psi)-\nu(\psi)  \right\},$$
where $\tilde{\psi}(x):=\inf_{Q\in \Pi(\delta_x,\cdot)}\E^Q\left[\int_0^1g(t,W(t),q^Q(t),a^Q(t))\,dt + \Psi(X_1)\right]$. 
\end{theorem}

\begin{proof}
Define $C(x,p):=V(\delta_x,p)$. 

Remark that if $Q_1\in \Pi(\delta_x,p_1)$ and $Q_2\in \Pi(\delta_x,p_2)$, then for $\alpha\in[0,1]$ we have $\alpha Q_1+(1-\alpha)Q_2\in \Pi(\delta_x,\alpha p_1+(1-\alpha)p_2) $. As we will see in Theorem \ref{thm compact convex}, the function $\alpha^g(\cdot)$ is convex. From these facts it follows that $C(x,\cdot)$ is convex for each $x$ fixed. Further, Theorem \ref{thm compact convex} also shows that $\alpha^g(\cdot)$ is coercive and lower semicontinuous, which then implies that $C(x,\cdot)$ is lower semicontinuous for each $x$ fixed.

A standard measurable selection argument shows that
$$V(\mu,\nu)=\inf_{\pi\in Cpl(\mu,\nu)}\int V(\delta_x,\pi^x)\mu(dx)=\inf_{\pi\in Cpl(\mu,\nu)}\int C(x,\pi^x)\mu(dx),$$
where we wrote $Cpl(\mu,\nu)$ for the set of measures in $\R^d\times\R^d$ with the given marginals. Applying the duality \eqref{eq:duality} we deduce
$$V(\mu,\nu)= \sup_\psi\left \{ \int R_C\psi(x)\mu(dx)-\nu(\psi)  \right \},$$
where
\begin{align*}
R_C\psi(x) &:= \inf_p\{p(\psi)+C(x,p)\}\\
&= \inf_{Q\in\Pi(\delta_x,\cdot)}\E^Q\left[\int_0^1g(t,W(t),q^Q(t),a^Q(t))\,dt + \Psi(X_1)\right]\\
&=\tilde\psi(x).
\end{align*}
\end{proof}

For the previous result we employed:

\begin{theorem}\label{thm compact convex}
The functional $\alpha^g$ is convex, lower semicontinuous with respect to weak convergence, and coercive in the sense that $\{Q:\alpha^g(Q)\leq c\}$ is weakly compact for each $c\in\R$.
\end{theorem}

In order to prove this we need the following auxiliary result first:

\begin{lemma} \label{le:H1-tightness}
Suppose $(q_n)_n $ is a sequence of $L^1([0,1],dt;\mathbb R^d)$-valued random variables possibly defined in different probability spaces, and call $A_n(t) := \int_0^tq_n(s)ds$. In the same space where $q_n$ is defined we are given a further $\C$-valued random variable $M_n$ such that $M_n(0)=0$, $M_n$ is a martingale wrt.\ its completed filtration, and such that $a_n(t) := \frac{\langle M_n\rangle(t)}{dt} $ exists a.s. Finally assume the existence of $c > 0$ such that, for all $n$,
\begin{align}
\E \left[\int_0^tg(t,M_n(t)+A_n(t),q_n(t),a_n(t))dt \right ] \le c. \label{def:H1-tightness-bound}
\end{align}
Then there exist an $L^1([0,1],dt;\mathbb R^d)$-valued random variable $A$, a $\C$-valued random variable $M$, and subsequences $A_{n_k}$ and $M_{n_k}$ such that 
\begin{enumerate}
\item $A_{n_k}$ converges in law in $\C$ to $A$,
\item  $M_{n_k}$ converges in law in $\C$ to $M$,
\item $A(t) = \int_0^tq(s)ds$, some $L^1([0,1],dt;\mathbb R^d)$-valued random variable $q$,
\item $\langle M \rangle(t) = \int_0^t a dt$, some $\SS_+^d$-valued process $a$ with $\E\left[\int_0^1|a(t)|dt\right ]<\infty$
\item the following inequality holds:
\begin{align}
&\E\left [\int_0^tg(t,M(t)+A(t),q(t),a(t))dt \right ] \notag \\ \le & \liminf_{k\rightarrow\infty}\E\left[\int_0^tg(t,M_{n_k}(t)+A_{n_k}(t),q_{n_k}(t),a_{n_k}(t))dt \right ]\label{def:H1-tightness-liminf}
\end{align}
\end{enumerate}
In particular, the laws of $(A_n)$ and $(M_n)_n$ form tight sequences.
\end{lemma}

{The following proofs follow very closely the arguments in \cite{BaLaTa18}. As a small technical improvement over \cite{TaTo13}, we observe that the coercivity condition \eqref{eq:classical kantorovich} assumed here, is weaker than the analogue in the cited paper.}

\begin{proof}[Proof of Lemma \ref{le:H1-tightness}]$ $

We first check tightness. {If the final part of Assumption \ref{ass:TT}(3) holds, this is trivial. Otherwise, by} \eqref{eq:coer_g}, for each $r > 0$ we may find $N > 0$ such that $g(t,x,q,a) \ge r|q|+r|a|$ whenever $|q|\vee |a| \ge N$. Moreover, there exists $b\ge0$ such that $g(t,x,q,a)\ge -b$. In particular, for all $(t,q)$ we have $|q|+|a| \le 2N + \frac 1r(g(t,x,q,a) + b)$.
Hence, for $0 \le s < t \le 1$, 
\begin{align*}
|A_n(t)-A_n(s)| &\le \int_s^t|q_n(u)|du \\& \le \frac{1}{r}\int_s^t (g(u,W(u),q_n(u),a_n(u)) + b) \,du + 2N(t-s) \\
	&\le \frac{1}{r}\int_0^1 g(u,W(u),q_n(u),a_n(u)) \,du + \frac{b}{r} + N(t-s).
\end{align*}
Hence, for any $\delta_n \downarrow 0$, \eqref{def:H1-tightness-bound} yields
\begin{align*}
\limsup_{n\rightarrow\infty} \sup_\tau \E|A_n(\tau+\delta_n) - A_n(\tau)| &\le \limsup_{n\rightarrow\infty}\left(\frac{c+b}{r} + N\delta_n \right) = \frac{c+b}{r},
\end{align*}
where the $\sup_\tau$ is over all stopping times with values in $[0,1-\delta_n]$.
As $r > 0$ was arbitrary, this shows that 
\begin{align*}
\limsup_{n\rightarrow\infty} \sup_\tau \E|A_n(\tau+\delta_n) - A_n(\tau)| = 0,
\end{align*}
and from Aldous' criterion for tightness  \cite[Theorem 16.11]{Kallenberg} we conclude that $(A_n)$ is tight. The Cauchy-Schwartz inequality and similar calculations allow to conclude that 
$$\E[|M_n(\tau+\delta)-M_n(\tau)|]\leq \sqrt{N\delta + \frac{c+b}{r}},$$
so as before $(M_n)$ is a tight sequence. %By continuous mapping, for each $1\leq i,j\leq d$ the sequence $(M_n^iM_n^j)$ is tight, and consequently $(\langle M_n\rangle)$ is likewise tight. 
Furthermore, $(M_n)$ is in fact precompact in the 1-Wasserstein space $\mathcal W_1(\C)$ of measures on $\C$ which integrate the supremum of the norm of the path of the canonical process. This follows by showing that
$$\lim_{K\to\infty} \sup_n \E\left[ \sup_{t\leq 1}|M_n(t)|1_{\sup_{t\leq 1}|M_n(t)|\geq K} \right]=0,$$
which is a consequence of Cauchy-Schwartz, Doob's inequality, and Assumption \ref{ass:TT}(3).

Passing to a subsequence and applying Skorokhod's representation, let us now assume that there exists continuous process $A$ and $M$ such that $A_n \rightarrow A$ and $M_n\rightarrow M$ almost surely (in $\C$), with all processes defined on some common probability space $(\Omega,\F,\PP)$. The process $M$ is in fact a martingale thanks to $\mathcal W_1(\C)$-precompactness. From Assumption \ref{ass:TT}(3) and a standard argument as in the de la Vall\'ee Poisson Theorem, we conclude that $\{q_n : n \in \N\} \subset L^1 := L^1([0,1] \times \Omega, \, dt \otimes d\PP)$ is uniformly integrable and thus weakly precompact. Similarly, using further that $\{a_n : n \in \N\}$ is weakly precompact in the Bochner space $L^1_d:=L^1([0,1] \times \Omega, \, dt \otimes d\PP;\R^{d\times d})$ of matrix-valued integrable processes if and only if  $\{|a_n| : n \in \N\}$ is uniformly integrable (cf.\ \cite{Diestel}), we deduce that $\{a_n : n \in \N\}$ is weakly precompact.  By passing to a further subsequence, we may now assume that $q_n \rightarrow q$ weakly in $L^1$, that $a_n\rightarrow a$ weakly in $L^1_d$, and that $a$ is almost surely $\SS_+^d$-valued. Because $g$ is bounded from below and lower semicontinuous in its last three variables, the map $(\bar X,\bar q,\bar a) \mapsto \E\int_0^1 g(t,\bar X(t),\bar q(t),\bar a(t))dt$ is lower semicontinuous in the norm topology of $\C\times L^1\times L^1_d$, by Fatou's lemma. Because it is also convex in the last two, this map is therefore weakly lower semicontinuous when $L^1\times L^1_d$ is given the weak topology. This yields \eqref{def:H1-tightness-liminf}.
By dominated convergence, it holds for each bounded random variable $Z$ that
\begin{align*}
\E[ZA(t)] = \lim_{n\rightarrow\infty}\E[ZA_n(t)] = \lim_{n\rightarrow\infty}\E\left[Z\int_0^tq_n(s)ds\right] = \E\left[Z\int_0^tq(s)ds\right].
\end{align*}
Hence $A(t) = \int_0^tq(s)ds$ a.s.\ for each $t$, and by continuity we have $A(\cdot) = \int_0^\cdot q(s)ds$ a.s.  A similar argument shows that for each $1\leq i,j\leq d$ and $0\leq s\leq t\leq 1$ and $Z$ measurable up to time $s$, we have
\begin{align*}0=&\lim_{n\rightarrow\infty}\E\left[ Z\left(M_n^iM_n^j(t)-M_n^iM_n^j(s)-\int_s^ta_n^{i,j}(r)dr\right)  \right ]\\=& \E\left[ Z\left(M^iM^j(t)-M^iM^j(s)-\int_s^ta^{i,j}(r)dr\right)  \right ] ,
\end{align*}
from which $\langle M \rangle(\cdot)=\int_0^\cdot a(r)dr$.
\end{proof}

\begin{proof}[Proof of Theorem \ref{thm compact convex}]$ $

\textbf{Convexity}: Let $\lambda\in [0,1]$, and fix $Q_0,Q_1\in \S$. We work on the extended probability space $\C \times \{0,1\}$, and we write $(W,X)$ to denote the identity map on this space. We define a measure $M$ on $\C \times \{0,1\}$ by requiring that the second marginal of $M$ be $\lambda \delta_0 + (1-\lambda)\delta_1$, and the conditional law of $W$ given $X$ be $Q_X$. In particular, the first marginal of $M$ is precisely $Q:=\lambda Q_0 + (1-\lambda)Q_1$.
Abbreviate $q_i:=q^{Q_i}$ and $a_i:=a^{Q_i}$.
It easily follows that the process
\[
W(t) - \int_0^t q_X(s)ds
\]
defines an $M$-martingale with respect to the filtration $\overline\FF=(\overline\F_t)_{t \in [0,1]}$ defined by $\overline\F_t=\F_t \otimes \sigma(X)$ on the product space. Furthermore, the quadratic variation of $W(\cdot) - \int_0^\cdot q_X(s)ds$ has a density explicitly given by $t\mapsto a_X(t)$.  Now define the processes $q=(q(t))_{t \in [0,1]}$ and $a=(a(t))_{t\in[0,1]}$ respectively as the optional projections of the processes $(q_X(t))_{t \in [0,1]}$ and $(a_X(t))_{t \in [0,1]}$ on the filtration $\mathcal F$ generated by $W$. In particular,
\begin{align*}
q(t) = \E^M[q_X(t) \, | \, (W_s)_{s \le t}] & = \E^M[ {\bf 1}_{X=0}q_0(t) + {\bf 1}_{X=1}q_1(t) \, | \, (W_s)_{s \le t}  ],\\
a(t) = \E^M[a_X(t) \, | \, (W_s)_{s \le t}] & = \E^M[ {\bf 1}_{X=0}a_0(t) + {\bf 1}_{X=1}a_1(t) \, | \, (W_s)_{s \le t}  ],
\end{align*}
A few computations reveal that $W(\cdot) - \int_0^\cdot q(t)dt$
is still an $M$-martingale. On the other hand, since $W(\cdot) - \int_0^\cdot q_X(s)ds$ has $a_X$ as the density of its quadratic variation under $M$, so does $W$ itself. But then for all $i\leq i,j\leq d$, if $R$ is bounded and ${\mathcal F}_t$-measurable and $0\leq h \leq 1-t$, we have
\begin{align*}
0&=\E^M\left [ \left ( W^iW^j(t+h)-W^iW^j(t)-\int_t^{t+h} a^{i,j}_X(s)ds \right ) R \right] \\ &= \E^M\left [ \left ( W^iW^j(t+h)-W^iW^j(t)-\int_t^{t+h} a^{i,j}(s)ds \right ) R \right] \\
&= \E^Q\left [ \left ( W^iW^j(t+h)-W^iW^j(t)-\int_t^{t+h} a^{i,j}(s)ds \right ) R \right],
\end{align*}
recalling that $Q$ is the first marginal of $M$. Since the $M$ martingale $W(\cdot) - \int_0^\cdot q(t)dt$ is $\mathcal F$-adapted, it follows that it is a $Q$-martingale (when seen as living in the filtered probability space $(\C,\FF,Q)$) and the above display shows that the density of its quadratic variation is precisely $a$.  In summary, we conclude that $Q \in \S$, and that $q=q^Q$ as well as $a=a^Q$. Finally, using Jensen's inequality, we compute
\begin{align*}
& \lambda\alpha^g(Q_0) + (1-\lambda)\alpha^g(Q_1) \\ = & \lambda \E^{Q_0}\left[\int_0^1 g(t,W(t),q_0(t),a_0(t))dt\right]+(1-\lambda)\E^{Q_1}\left[\int_0^1 g(t,W(t),q_1(t),a_1(t))dt\right] \\ 
	=& \E^M\left[\int_0^1 g(t, W(t),q_X(t),a_X(t))dt\right] \\ 
	\ge & \E^M\left[\int_0^1 g(t,W(t), q(t),a(t))dt\right] \\ =& \E^{Q}\left[\int_0^1 g(t,W(t),q(t),a(t))dt\right] \\
	=& \alpha^g(Q).
\end{align*}

{
\textbf{Inf-compactness:} 
Let $c\in \R$ and $\Lambda_c := \{Q:{\alpha}^g(Q)\leq c\}$. It is convenient in this step and the next to define
\[
W^Q(t) := W(t) - \int_0^t q^Q(s)ds, \quad t \in [0,1],
\]
for $Q \in \S$, noting that $W^Q$ is a $Q$-martingale with volatility $a^Q$. Letting $A^Q(t) := \int_0^tq^Q(s)ds$, it follows from Lemma \ref{le:H1-tightness} that $\{Q \circ (A^Q)^{-1} : Q \in \Lambda_c\} \subset \P(\C)$ is tight. On the other hand, $\{Q \circ (W^Q)^{-1} : Q \in \Lambda_c\} $ is tight as well by the same argument. Since each marginal is tight, we deduce that $\{Q \circ (W^Q,A^Q)^{-1} : Q \in \Lambda_c\} \subset \P(\C \times \C)$ is tight. Finally, by continuous mapping, the set $\{Q \circ (W^Q + A^Q)^{-1} : Q \in \Lambda_c\} = \Lambda_c$ is tight.
}

{
\textbf{Lower semicontinuity:} Suppose $\{Q_n : n \in \N \} \subset \Lambda_c$ with $Q_n \rightarrow Q$ weakly for some $Q \in \P(\C)$. We must show that $Q$ belongs to $\Lambda_c$. Define the continuous process
\[
A_n(t): = \int_0^t q^{Q_n}(s)ds = W(t) - W^{Q_n}(t),
\]
for each $n$. As in the previous point, $\{Q_n \circ (W,W^{Q_n},A^n)^{-1} : n \in \N\}$ is tight. 
Relabelling a subsequence, suppose that $Q_n \circ (W,W^{Q_n},A^n)^{-1}$ converges weakly to the law of some $\C^3$-valued random variable $(X,B,A)$.
Using Lemma \ref{le:H1-tightness}, we may assume also that $A(t) = \int_0^tq(s)ds$ and $a(t):=\frac{\langle B\rangle(t)}{dt}$  satisfying
\[
\E\int_0^1g(t,X(t),q(t),a(t))dt \le \liminf \E^{Q_n}\int_0^1g(t,W(t),q^{Q_n}(t),a^{Q_n}(t))dt \le c.
\]
Clearly $W^{Q_n}$ is a martingale in the filtration of $(W,W^{Q_n},A^n)$, and hence $B$ is a martingale in the filtration of $(X,B,A)$. Finally, notice that
\[
X(t) = B(t) + A(t) = B(t) + \int_0^tq(s)ds,
\]
as the same relation holds in the pre-limit.
A standard argument shows that $X - \int_0^{\cdot}\widehat{q}(s)ds$ is a martingale with respect to the filtration of $X$, where $\widehat{q}$ is the optional projection of $q$ onto such filtration. Writing 
$$X(t)=B(t) + \int_0^t[q(s)-\widehat{q}(s)]ds+\int_0^{t}\widehat{q}(s)ds=: \tilde B(t)+\int_0^{t}\widehat{q}(s)ds,$$
we deduce that $\tilde B$ is an $X$-adapted martingale with density of quadratic variation $a$.
By convexity of $g(t,x,\cdot,\cdot)$, we have
\[
\E\int_0^1g(t,X(t),\widehat{q}(t),a(t))dt \le \E\int_0^1g(t,X(t),q(t),a(t))dt \le c.
\]
Recalling that $Q$ denoted the law of $X$, we conclude that $Q \in \Lambda_c$.
}

\end{proof}

\bibliography{joint_biblio}
\bibliographystyle{abbrv}
\end{document}